\documentclass[12pt,a4paper]{article}
\usepackage{amsmath,amscd,amsxtra,amsthm,amssymb,makeidx,graphics,booktabs,mathrsfs,titlesec,url}
\theoremstyle{plain}

\newtheorem{theorem}{Theorem}[section]
\newtheorem{lemma}[theorem]{Lemma}
\newtheorem{corollary}[theorem]{Corollary}
\newtheorem{proposition}[theorem]{Proposition}

\newtheorem{remark}[theorem]{Remark}
\newtheorem{example}[theorem]{Example}

\newtheoremstyle{citing}{3pt}{3pt}{\itshape}{0pt}{\bfseries}%
{.}{ }{\thmnote{#3}}\theoremstyle{citing}

\newcommand{\PSL}{\operatorname{PSL}}
\newcommand{\SL}{\operatorname{SL}}

\begin{document}
\title{Explicit Families Of Elliptic Curves With The Same Mod 6 Representations}
\author{Zexiang Chen}
\date{}
\maketitle
\begin{abstract}
Let $E$ be an elliptic curve over $\mathbb{Q}$. In this paper we study two certain modular curves which parameterize families of elliptic curves which are directly (resp. reverse) 6-congruent to $E$.
\end{abstract}

\section{Introduction and Notations}
\subsection{Introduction and General Background}
$\quad$ If $n$ is a positive integer and $E$ is an elliptic curve over $\mathbb{Q}$ with algebraic closure $\bar{\mathbb{Q}}$, let $E[n]$ be the $n$-torsion subgroup of $E$, which is the kernel of multiplication by $n$ on $E(\bar{\mathbb{Q}})$. Throughout the paper, each elliptic curve will be assumed to be defined over $\mathbb{Q}$.

Elliptic curves $E_1$ and $E_2$ are said to be $n$-congruent if $E_1[n]$ and $E_2[n]$ are isomorphic as Galois modules. They are directly $n$-congruent if the isomorphism $\phi : E_1[n] \cong E_2[n]$ respects the Weil pairing and reverse $n$-congruent if $e_n(\phi P,\phi Q) = e_n(P,Q)^{-1}$ for all $P,Q \in E_1[n]$ where $e_n$ is the Weil pairing.
The families of elliptic curves which are directly (resp. reverse) n-congruent to a given elliptic curve $E$ are parameterized by the
modular curve $Y_E(n) = X_E(n) \backslash \{\text{cusps}\}$ (resp. $Y^{-}_E(n)=X^{-}_E(n) \backslash \{\text{cusps}\}$). That means each point on $Y^{\pm}_E(n)$ corresponds to an isomorphism class $(F,\phi)$ where $F$ is an elliptic curve and $\phi$ is a Galois equivariant isomorphism $E[n] \to F[n]$.

Let $W_{E}(n)$ (resp. $W^{-}_{E}(n)$) be elliptic surface over $X_{E}(n)$ (resp. $X^{-}_{E}(n)$) whose fibers above points of $Y_{E}(n)$ (resp. $Y^{-}_{E}(n)$) are elliptic curves which are directly (resp. reverse) $n$-congruent to $E$.

It was shown in [S] that $X^{\pm}_E(n)$ are twists of the classical modular curves $X(n)$. For the case $n \le 5$, the modular curve $X(n)$ has genus zero and there are infinitely many rational points on $X_E(n)$ for $n \le 5$ and the structures of $X_E(n)$ were studied by Rubin and Silverberg in [RS1],[RS2] and [S]. The reverse case $X^{-}_E(n)$ was studied by Fisher using invariant theory [F1], [F2].

For $n=6$, the curve $X(6)$ has genus one and so $X_{E}(6)$ and $X^{-}_{E}(6)$ are both elliptic curves over $\bar{\mathbb{Q}}$.
A model for $X_{E}(6)$ was obtained independently by Papadopoulos in [P] and by Rubin and Silverberg in [RS3].
Equation of the $W_E(6)$ was computed by Roberts in [R] for some elliptic curve $E$ with specified $j$-invariant. In this paper we compute the equations for $X^-_E(6)$ and $W^{\pm}_E(6)$ for every elliptic curve $E$.

\begin{flushleft}
{\bf Acknowledgement}
\end{flushleft}
Most symbolic computations were done by MAGMA [MAG]. I would like to thank T.A.Fisher for some useful discussion and advice.
\subsection{Main Theorem}
\begin{theorem} Let $E: y^2=x^3+ax+b$ be an elliptic curve. A model for
$X^{-}_E(6)$ is given by intersection of 9 quadrics in $\mathbb{P}^5$ with coordinates $(x_1:\ldots:x_6)$
\begin{align*}
s_1&=-6x_1x_5 + 24ax_1x_6 - 6x^2_2 + 24ax_2x_3 - 6x_2x_4 + 24ax_3x_4\\
 &+ 72bx_4x_6 + 2ax^2_5 + 8a^2x_5x_6 +Dx^2_6,\\
s_2&=-6x_1x_3 + x_2x_5 + 2ax_2x_6 - 36bx^2_3 + 2ax_3x_5 + 16a^2x_3x_6 + x_4x_5\\
& + 2ax_4x_6 + 12abx^2_6,\\
s_3&=12ax_1x_3 + 18bx_1x_6 + 18bx_3x_4 - 2ax_4x_5 - 4a^2x_4x_6 + 3bx^2_5,\\
s_4&=-12ax_2x_3 - 18bx_2x_6 - 18bx_3x_5 - 3x^2_4 - ax^2_5 + 4a^2x_5x_6,\\
s_5&=3x^2_2 - 48a^2x^2_3 - 144abx_3x_6 - 36bx_4x_6 + ax^2_5 - 8a^2x_5x_6 + 16a^3x^2_6,\\
s_6&=-3x_1x_4 + 18bx_2x_3 - ax_2x_5 - 4a^2x_2x_6 - 4a^2x_3x_5 - 6abx_5x_6,\\
s_7&=-108bx_1x_3 + 6ax^2_2 - 24a^2x_2x_3 + 18bx_2x_5 - 36abx_4x_6 - 2a^2x^2_5\\
& - 8a^3x_5x_6-aDx^2_6,\\
s_8&=3x_1x_2 - 72abx^2_3 - 216b^2x_3x_6 + ax_4x_5 + 8a^2x_4x_6 - 12abx_5x_6 + 24a^2bx^2_6,\\
s_9&=36x^2_1 + 12ax^2_2 + 12ax^2_4 + 4a^2x^2_5 + Dx_5x_6.\\
\end{align*}
where $D$ is the discriminant of $E$.
\end{theorem}
By using Theorem 1.1 we will then obtain a simpler model for $X^{-}_E(6)$
\begin{corollary} Let $E:y^2=x^3+ax+b$ be an elliptic curve. Then $X^{-}_E(6)$ is birational to the following curve defined by two equations in $\mathbb{A}^3$ with coordinates $(x,y,z)$:
\begin{align*}
f=&z^3-(36ax^2+12a^2)z+216bx^3-144a^2x^2-216abx\\
&-(16a^3+216b^2)+27y(64abx+96b^2)/D,\\
g=&y^2-D(ax^4+6bx^3-2a^2x^2-2abx+(-a^3/3-3b^2))\\
\end{align*}
where $D$ is the discriminant of $E$.
\end{corollary}

\section{Preliminary Material}
$\quad$In this section we will briefly give some material and preliminary results which can be found in the references.
\subsection{Modular Curves}
$\quad$Let $n \ge 2$ be an integer. Let $Y(n)$ denote the classical modular curve which parametrizes
isomorphism classes of pairs $(E, \phi)$ where $E$ is an elliptic curve and
$$\phi:  \mathbb{Z}/n\mathbb{Z} \times \mu_n \to E[n]$$
is an isomorphism such that
$$e_n(\phi(a_1,\zeta_1),\phi(a_2,\zeta_2))=\zeta^{a_1}_2/\zeta^{a_2}_1.$$
\begin{lemma}
Equivalently, $Y(n)$ parametrizes triples $(E,P,C)$ where $E$ is an elliptic curve, $P$ is a point of exact order $n$
and $C$ is a cyclic subgroup of order $n$ which does not contain $P$.
\end{lemma}
\begin{proof} Given a triple $(E,P,C)$ define $\phi(a,\zeta^b)=aP+bQ$ where $Q \in C$ such that $e_n(P,Q)=\zeta$.
Conversely, given $(E,\phi)$ define $P$ to be $\phi(1,1)=P$, $\phi(0,\zeta)=Q$ and $C=\langle Q \rangle$.
\end{proof}

Let $X(n)$ be the compactification of $Y(n)$. We identify $X(1)$ with $\mathbb{P}^1$ by $E \mapsto j(E)$. Let $W_n$ be the total space. That is
a quasi-projective surface defined over $\mathbb{Q}$ with a projection morphism
$$\pi_n:  W_n \to Y(n)$$
and a zero section $Y(n) \to W_n$ both defined over $\mathbb{Q}$
with $n^2$ sections defined over $\bar{\mathbb{Q}}$ of order dividing $n$, and such that the fibers
of $\pi_n$ correspond to the triples $(E,P,C)$ classified by $Y(n)$.
\begin{lemma}
The forgetful morphism $X(n) \to X(1)$
induced by $(E,\phi) \mapsto E$ has degree
$|\PSL_2(\mathbb{Z}/n\mathbb{Z})|$ where
$\PSL_2(\mathbb{Z}/n\mathbb{Z})=\SL_2(\mathbb{Z}/n\mathbb{Z})/\{\pm I \}$.
\end{lemma}
\begin{proof}
For each $\alpha \in \SL_2(\mathbb{Z}/n\mathbb{Z})$, $\alpha$ acts on $Y(n)$ by
$\alpha\circ (E,\phi)=(E,\alpha \circ \phi)$ and $-I$ acts trivially because it is induced by $[-1]$ which
is an automorphism of $E$. So $\PSL_2(\mathbb{Z}/n\mathbb{Z})$ acts on $Y(n)$ and the action
extends to $X(n)$. Finally, the quotient map corresponds to $(E,\phi) \to E$ which is just taking the $j$-invariant
and so the degree of the forgetful morphism is the same as $|\PSL_2(\mathbb{Z}/n\mathbb{Z})|$.
\end{proof}

For $n=6$, the curve $X(6)$ has genus one and
\begin{theorem} A model for $Y(6) \subset \mathbb{A}^2$ is given by
$$2v^2\tau^2=v+\tau, v(8v^3+1)(v^3-1) \neq 0.$$
Setting $X=2v, Y=4v^2\tau-1$ and taking compactification, we obtain a model for $X(6): Y^2=X^3+1$.
\end{theorem}
\begin{proof} See [P] Theorem 3.
\end{proof}
\subsection{Modular Elliptic Curves}
$\quad$Let $E$ be an elliptic curve. We state the general theoretical construction for $X_E(n)$ which can be found in [S]
\begin{theorem}
Let $E$ be an elliptic curve and $V=E[n]$, viewed as $G_\mathbb{Q}$-module. Define a bilinear
pairing $\langle , \rangle$ on $\mathbb{Z}/n\mathbb{Z} \times \mu_n$ by
$$\langle (a_1,\zeta_1),(a_2,\zeta_2) \rangle =\zeta^{a_2}_1/\zeta^{a_1}_2.$$
Now fix an isomorphism $\phi: \mathbb{Z}/n\mathbb{Z} \times \mu_n \to V$ such that the above pairing is compatible
with the Weil pairing under $\phi$. Then the cocycle
$$\tau \mapsto \phi^{-1} \tau(\phi)$$
take values in Aut$(\mathbb{Z}/n\mathbb{Z} \times \mu_n,\langle,\rangle)$ which is the set of automorphisms
of $\mathbb{Z}/n\mathbb{Z} \times \mu_n$ which preserve the bilinear pairing.

By Lemma 2.1 the identification induces an inclusion
$$\text{Aut} (\mathbb{Z}/n\mathbb{Z} \times \mu_n,\langle,\rangle) \hookrightarrow \text{Aut}(W_n)$$
and also there is a natural map Aut$(W_n) \to$ Aut$(X(n))$. Thus, the cocycle above induces cocycles
$c$ and $c_0$, taking values in Aut$(W_n)$ and Aut$(X(n))$ respectively. Then by general theory of twist,
we obtain curves $W$ and $X$, and induced isomorphisms $\psi$ and $\psi_0$ defined over $\bar{\mathbb{Q}}$
together with a projection map $\pi: W \to X$ defined over $\mathbb{Q}$ such that the following diagram commutes
\begin{center}
$\begin{CD}
W @>\psi>> W_n\\
@VV\pi V @VV \pi_n V\\
X @>\psi_0>> X(n)
\end{CD}$
\end{center}
Then the curve $X$ is a model for $X_E(n)$ over $\mathbb{Q}$.
\end{theorem}
\begin{proof}
See [S], page 449-450.
\end{proof}

We now state the results for $X_E(2)$ in [RS1] and the results for $X^{\pm}_E(3)$ in [F1],[F2]. Note that $X^{-}_E(2)$ is the same as $X_E(2)$.
\begin{theorem} Let $E: y^2=x^3+ax+b$ be an elliptic curve. Then the family of elliptic curves parameterized by $Y_E(2)$
are $F_{u,v}$:
$$y^2 = x^3 + 3(3av^2 + 9buv - a^2u^2)x + 27bv^3 -18a^2uv^2 - 27abu^2v - (2a^3 + 27b^2)u^3,$$
where $(u:v) \in \mathbb{P}^1(\mathbb{Q})$.
Further, $\Delta_{F_{u,v}}=3^6(v^3 + au^2v + bu^3)^2\Delta_E$.
\end{theorem}
\begin{proof} See [RS2] Theorem 1.
\end{proof}
\begin{theorem} Let $E: y^2=x^3+ax+b$ be an elliptic curve. Let $c_4=-a/27$,$c_6=-b/54$ and
\begin{align*}
\mathfrak{c}_4(\lambda,\mu)
&=c_4\lambda^4+4c_6\lambda^3\mu+6c^2_4\lambda^2\mu^2+4c_4c_6\lambda \mu^3 - (3c^3_4
-4c^2_6)\mu^4,\\
\mathfrak{c}_6(\lambda,\mu)
&=c_6\lambda^6+6c^2_4\lambda^5\mu+15c_4c_6\lambda^4\mu^2+20c^2_6\lambda^3\mu^3
+15c^2_4c_6\lambda^2\mu^4\\
&+6(3c^4_4- 2c_4c^2_6)\lambda\mu^5+(9c^3_4c_6 - 8c^3_6)\mu^6,\\
\mathfrak{c}^*_4(\lambda,\mu)
&=-4(\lambda^4-6c_4\lambda^2\mu^2 -8c_6\lambda\mu^3-3c^2_4\mu^2)/(c^3_4-c^2_6),\\
\mathfrak{c}^*_6(\lambda,\mu)
&=-8\mathfrak{c}_6(\lambda,\mu)/(c^3_4-c^2_6)^2.\\
\end{align*}
Then the family of elliptic curves parameterized by $Y_E(3)$ and $Y^{-}_E(3)$ are
$$\mathcal{E}_{\lambda,\mu}:y^2=x^3-27\mathfrak{c}_4(\lambda,\mu)x-54\mathfrak{c}_6(\lambda,\mu).$$
and
$$\mathcal{E}'_{\lambda,\mu}:y^2=x^3-27\mathfrak{c}^*_4(\lambda,\mu)x-54\mathfrak{c}^*_6(\lambda,\mu).$$
\end{theorem}
\begin{proof} See [F2] Theorem 1.1.
\end{proof}
\section{Explicit Families of Directly 6-Congruent Elliptic Curves}
$\quad$In this section we give a new method to find a model for $X_E(6)$ and we will use the same method to compute $X^{-}_E(6)$ in the next section.
\subsection{The Action of Special Projective Linear Group}
\begin{lemma} Fix an isomorphism $\eta_n: X_E(n) \to X(n)$.
Let $\alpha \in$ $\PSL_2(\mathbb{Z}/n\mathbb{Z})$ then the action of $\alpha$ on $X_E(n)$ is given by
$$\alpha P=\eta^{-1}_n(\alpha(\eta_n(P))), P \in X_E(n).$$
Further, the forgetful morphism $X_E(n) \to X(1)$ is the quotient map by the action defined above.
\end{lemma}
\begin{proof} This is an immediate consequence of Theorem 2.4.
\end{proof}
The action in the reverse case is defined in a similar way as above. We now focus on the case $n=2$ and $3$ and state the following standard fact.
\begin{lemma}
We have $\PSL_2(\mathbb{Z}/2\mathbb{Z}) \cong S_3, \PSL_2(\mathbb{Z}/3\mathbb{Z}) \cong A_4$
and exact sequences
$$0 \to H \to \PSL_2(\mathbb{Z}/2\mathbb{Z}) \to C_2 \to 0,$$
$$0 \to H' \to \PSL_2(\mathbb{Z}/3\mathbb{Z}) \to C_3 \to 0$$
where $H \cong C_3$ and $H' \cong C_2 \times C_2$.
\end{lemma}

For $m|n$, viewing $\PSL_2(\mathbb{Z}/m\mathbb{Z})$ as a subgroup of
$\PSL_2(\mathbb{Z}/n\mathbb{Z})$ we obtain the forgetful morphism induced by  the action of $\PSL_2(\mathbb{Z}/m\mathbb{Z})$
$$X(n) \to X(n/m), X_E(n) \to X_E(n/m).$$
Thus, we have forgetful morphisms
$$\psi_2: X_E(6) \to X_E(2), \psi_3: X_E(6) \to X_E(3)$$
induced by the action of $\PSL_2(\mathbb{Z}/3\mathbb{Z})$ and $\PSL_2(\mathbb{Z}/2\mathbb{Z})$ respectively.
Therefore, $X_E(6)$ can be obtained by taking the fiber product of $X_E(3)$ and $X_E(2)$. Explicitly, we have the following commutative diagram:
\begin{center}
$\begin{CD}
X_E(6) @>\psi_2>> X_E(2)\\
@VV\psi_3V @VVf_3V\\
X_E(3) @>f_2>> X(1)
\end{CD}$
\end{center}
where $f_3,f_2$ are the forgetful morphisms from $X_E(3)$ (resp. $X_E(2)$) to $X(1)$.
\subsection{A Commutative Diagram}
$\quad$Throughout this subsection, we will work over $\bar{\mathbb{Q}}$.
By Lemma 3.2, the morphism $\psi_3$ factors through $\chi_3$ and $\rho_3$ where $\chi_3$ is the quotient morphism induced by the action of $H$ and so we have
\begin{center}
$\begin{CD}
X_E(6) @> \chi_3>> X @> \rho_3>> X_E(3)\\
\end{CD}$
\end{center}
Similarly, the morphism $\psi_2$ factors through
\begin{center}
$\begin{CD}
X_E(6) @> \chi_2>> Y @> \rho_2>> X_E(2)\\
\end{CD}$
\end{center}
where $\chi_2$ is the morphism obtained by quotient out the action of $H'$.

Further, we can take the quotient of $Y$ by the action of $H'$ (with quotient map $\alpha$) and the quotient of $X$ by the action of $H$ (with quotient map $\beta$). Thus we obtain the following commutative diagram
\begin{center}
$\begin{CD}
X_E(6) @>\chi_2>> Y\\
@VV\chi_3 V @VV\alpha V\\
X@>\beta >> Z
\end{CD}$
\end{center}

We now establish some properties of the diagram.
\begin{proposition} The curves $X,Y,Z$ all have genus one.
\end{proposition}
\begin{proof}
Fix an isomorphism $\eta_6:X_E(6) \to X(6)$. As $H,H'$ are normal subgroups we conclude the extensions of function fields of the curves in the above commutative diagram are Galois.

We show that $\chi_2$ is unramified. Suppose $P \in X_E(6)$ is a ramification point. Then it will be ramified under the map
$\psi_3: X_E(6) \to X_E(3)$
and so $\eta_6(P)$ will be ramified under the
forgetful map $X(6) \to X(3)$. By standard fact about modular curves and tower law the only ramification points are the cusps of $X(6)$ with ramification degree $2$. This shows that $P$ must have ramification degree $2$
which is a contradiction as $2$ is prime to the degree of $\chi_2$.
Therefore $g(X)=1$ by applying Riemann-Hurwitz.

Similarly as the only ramification points under $X(6) \to X(2)$ are cusps with
ramification degree $3$, we conclude that
the ramification degree of any points $P$ under $\chi_2: X_E(6) \to Y$ is a factor of $3$. Since $3$ is prime to the degree of $\chi_2$ we conclude $\chi_2$ is unramified. Applying Riemann-Hurwitz we conclude that $g(Y)=1$.

Finally, we use the above commutative diagram to conclude both $\alpha$ and
$\beta$ are unramified because $\chi_2,\chi_3$ are unramified and the degrees of the maps $X \to Z$ and $Y \to Z$ are coprime.
Then apply Riemann-Hurwitz, we conclude that $g(Z)=1$.
\end{proof}
\subsection{Galois Theory and Function Fields}
$\quad$We now compute explicit equations for $X,Y$ and $Z$.
As $X_E(2)$ and $X_E(3)$ are both
isomorphic to $\mathbb{P}^1$ over $\mathbb{Q}$ so we identify the function field of $X_E(2)$ with $\mathbb{Q}(u)$
and the function field of $X_E(3)$ with $\mathbb{Q}(\lambda)$, in the sense that in Theorem 2.5 we take the affine piece of $\mathbb{P}^1$ with coordinate $(u:v)$ by setting $v=1$ and in Theorem 2.6 we take the affine piece
of $\mathbb{P}^1$ with coordinate $(\lambda:\mu)$ by setting $\mu=1$.

Fix an elliptic curve $E$. Suppose $F$ is an elliptic curve which is directly 3-congruent to $E$. If $F$ is also 2-congruent to $E$,
then by Theorem 2.5, the ratio $\Delta_F/\Delta_E$ is a rational square. Thus, we consider the modular curve which parameterizes the families of elliptic curves that are directly 3-congruent to $E$ whose discriminant differs from $\Delta_E$ by a rational square. Then this curve is given by a quotient of $X_E(6)$, with function fields $\mathbb{Q}(\lambda,y)$ such that
$$\Delta_E y^2=\Delta_{\mathcal{E}_{\lambda,\mu}}$$
where $\mathcal{E}_{\lambda,\mu}$ is the elliptic curve as in Theorem 2.6. We denote the families of elliptic curves
parameterized by this modular curve by $\mathcal{E}_{y,\lambda,\mu}$. By considering the degree of extension of function fields, this modular curve is isomorphic to $X$ because both of them are quotient of $X_E(6)$ by the unique subgroup of order $3$ inside
$\PSL_2(\mathbb{Z}/2\mathbb{Z})$.
We have the following proposition
\begin{proposition} Let $E:y^2=x^3+ax+b$ be an elliptic curve. Then a model of $X$ in weighted projective space with coordinate $(\lambda:\mu:y)$ is
$$C_X: y^2=\lambda^4+2a\lambda^2\mu^2+4b\lambda\mu^3-\frac{1}{3}a^2\mu^4$$
By taking affine pieces, we obtain the forgetful morphism
$$C_X \to X_E(3), \quad (\lambda:1:y) \mapsto \lambda/3  \quad \text{ and } \quad (1:0:\pm 1) \mapsto (1:0).$$
\end{proposition}
\begin{proof}
From the modular interpretation of $X$ we described above we compute
$$X:\Delta_E y^2=\Delta_{\mathcal{E}_{\lambda,\mu}}
=\Delta_E\left(\lambda^4+\frac{2a}{9}\lambda^2\mu^2+\frac{4}{27}b\lambda\mu^3-\frac{1}{243}a^2\mu^4\right)^3.$$
Writing above equation in the form
$$\frac{81y^2}{\left(\lambda^4+\frac{2a}{9}\lambda^2\mu^2+\frac{4}{27}b\lambda\mu^3-\frac{1}{243}a^2\mu^4\right)^2}
=(3\lambda)^4+2a(3\lambda^2)\mu^2+4b(3\lambda)\mu^3-\frac{1}{3}a^2\mu^4$$
we see the curve $X$ is isomorphic to $C_X$.

Finally, by construction the morphism $X \to X_E(3)$ is given by $\mathcal{E}_{y,\lambda,\mu} \mapsto \mathcal{E}_{\lambda,\mu}$
and the isomorphism $C_X \to X$ is given by
$$(\lambda:\mu:y) \mapsto \left(\frac{\lambda}{3}:\mu: \frac{y}{9}(\lambda^4+\frac{2}{9}a\lambda^2\mu^2+\frac{4}{27}b\lambda\mu^3-\frac{1}{243}a^2\mu^4)\right)$$
and so we obtain the morphism $C_X \to X_E(3)$ as described in the proposition.
\end{proof}

The method of finding a model of $Y$ is similar. Let $F$ be an elliptic curve with
the same mod 2 representation as $E$. If $F$
is also directly 3-congruent to $E$, then it corresponds to $\mathcal{E}_{\lambda,\mu}$ for some $(\lambda:\mu) \in X_E(3)$.
Thus, by a direct computation, we have
$$\Delta_F=\Delta_{\mathcal{E}_{\lambda,\mu}}=\Delta_E\left(\lambda^4+\frac{2a}{9}\lambda^2\mu^2+\frac{4}{27}b\lambda\mu^3
-\frac{1}{243}a^2\mu^4\right)^3,$$
and hence $\Delta_F$ differs from $\Delta_E$ by a rational cube.
Therefore we conclude that there is a modular curve which is a quotient of $X_E(6)$ which
parameterizes families of elliptic curves which are $2$-congruent to $E$ whose discriminant differs from $\Delta_E$ by a rational cube. The function field of the modular curve is $\mathbb{Q}(u,y)$ such that
$$y^3\Delta_E=\Delta_{F_{u,v}}$$
where $F_{u,v}$ is the elliptic curve as in Theorem 2.5.
We denote the families of elliptic curve parameterized by this modular curve by $\mathcal{F}_{y,u,v}$.
Further $Y$ is isomorphic to this modular curve as both of them are the quotient of $X_E(6)$ by the unique subgroup of order $4$ in $\PSL_2(\mathbb{Z}/3\mathbb{Z})$.
\begin{proposition}
Let $E: y^2=x^3+ax+b$ be an elliptic curve. Then a model of $Y \subset \mathbb{P}^2$ is
$$C_Y: y^3=v^3+au^2v+bu^3$$
and the morphism $C_Y \to X_E(2)$ is $(u:v:y) \mapsto (u:v)$.
\end{proposition}
\begin{proof} From the modular interpretation of $Y$ above we compute
$$y^3\Delta_E=\Delta_{F_{u,v}}=3^6(v^3 + au^2v + bu^3)^2\Delta_E$$
and so
$$\left(\frac{9(v^3+au^2v+bu^3)}{y}\right)^3=v^3+au^2v+bu^3.$$
This is clearly birational to $C_Y$ by
$$(u:v:y) \to (u:v:9(v^3+au^2v+bu^3)/y).$$
The forgetful morphism $Y \to X_E(2)$ is $\mathcal{F}_{y,u,v} \mapsto F_{u,v}$
and so $C_Y \to X_E(2)$ given by $(u:v:y)\mapsto (u:v)$.
\end{proof}

The following describes a geometric interpretation of the action of $H'$.
\begin{lemma}
We identify $X$ with its Jacobian $J_X$ over $\bar{\mathbb{Q}}$. Then the
action of $H'$ on $J_X$ is translation by 2-torsion points of $J_X$. Further, the quotient map
$J_X \to Z$ by the action of $H'$ is multiplication-by-2 map in the sense that we can identify $J_X \cong Z$ over $\bar{\mathbb{Q}}$.
\end{lemma}
\begin{proof} Let $J_X,J_Z$ be the Jacobian of $X,Z$ respectively.
We have shown the map $X \to Z$ is unramified and so is $J_X \to J_Z$. Thus the action of $H'$ on $J_X$ does not have any fix point. So for each non-trivial element $f \in H'$, $O$ is not in the image of the morphism $f-1$ and so $f-1$ not surjective. Therefore it is constant and so $f$ acts as translation on $J_X$.
Since $f$ has order $2$ so the action is translation by 2-torsion points
and so the kernel of the quotient map $J_X \to J_Z$
is $J_X[2]$.
By Proposition 4.12 of [AEC],
there is a unique elliptic curve (up to $\bar{\mathbb{Q}}$-isomorphism) $E'$ and a separable isogeny
$$\phi: J_X \to E'$$
such that $\ker\phi=J_X[2]$ and so by uniqueness $E' \cong J_Z$. But the multiplication-by-2 map on $J_X$ also has kernel $J_X[2]$ and so by uniqueness $J_X \cong J_Z$.
\end{proof}

\begin{lemma} The curve $Z$ has a $\mathbb{Q}$-rational point and thus it is an elliptic curve. In particular,
$Z$ is the Jacobian of $X$. A model of $Z \subset \mathbb{P}^2$ is given by
$$y^2z=x^3-27\Delta_Ez^3.$$
\end{lemma}
\begin{proof} Take the model $C_X$ for $X$ which is given by
$$y^2=\lambda^4+2\lambda^2\mu^2+4b\lambda\mu^3-\frac{1}{3}a^2\mu^4=p(\lambda,\mu)$$
and let $t_i,i=1,2,3,4$ be the roots of the polynomial $p(\lambda,1)$. The 2-torsion
points of $J_X$ corresponds to $(t_i,0),i=1,2,3,4$ and by Lemma 3.6 the quotient map $\beta$ sends them to the same image, say $T=\beta((t_i,0))$.

For each $\sigma \in G_\mathbb{Q}$, as $\beta$ is defined over $\mathbb{Q}$ and $\sigma$ permutes the $2$-torsion points so we conclude that
$T$ is fixed by $\sigma$. Therefore $T$ is a $\mathbb{Q}$-rational point. Thus $Z$ is an elliptic curve with identity point $T$ and by Lemma 3.6 we conclude $Z$ is the Jacobian of $X$. The model of $Z$ in the statement can be obtained by a direct computation using the standard formula which can be found in 3.1 of [AKM$^3$P], where we use the model $C_X$ for the curve $X$.
\end{proof}
\begin{remark} Following the same argument as in Lemma 3.6 we conclude $X_E(6)$ and $Y$ have the same Jacobian. Recall $X_E(6)$ always admits a rational point which corresponds to the class $(E,[1])$. Lemma 3.7 is obvious with the observation $X$ and $Y$ are both in fact elliptic curves. We did not use this fact to prove the lemma because we will need an analogue of the lemma in the next section where $X$ and $Y$ do not necessarily have a rational point.
\end{remark}
An immediate consequence of the above remark is
\begin{theorem} A model of $X_E(6) \subset \mathbb{P}^2$ is given by
$$y^2z=x^3+\Delta_Ez^3.$$
\end{theorem}
\begin{proof} This follows from a direct computation of the Jacobian of $Y$ using the formula in 3.2 of [AKM$^3$P] where we use the model $C_Y$ for the curve $Y$.
\end{proof}
Note this result was obtained by Rubin and Silverberg [RS3] and Papadopoulos [P] using different methods.
\subsection{Explicit Equations}
$\quad$We will compute the morphism $X_E(6) \to X_E(3)$ explicitly and hence find the families of elliptic curves parameterized by $Y_E(6)$.
\begin{lemma} The map $\chi_3: X_E(6) \to X$ is 3-isogeny over $\bar{\mathbb{Q}}$.
\end{lemma}
\begin{proof} Over $\bar{\mathbb{Q}}$ any morphism between genus one curves is an isogeny and the map $\chi_3$ has
degree $3$ and so it is a 3-isogeny.
\end{proof}
We now need to give an explicit equation of the map $\chi_3$. Recall from Proposition 3.4 that the map $\rho_3: C_X \to X_E(3)$ is
$$(\lambda,y) \mapsto \lambda/3, \quad  \text{if } \mu=1, \quad \text {and } \quad (1:0:\pm 1) \mapsto (1:0).$$
and so we need to compute the morphism $X_E(6) \to C_X$. Recall from Theorem 3.9 our convention for the model
of $X_E(6)$ is we identify the point at infinity $(0:1:0)$ as the isomorphism class $(E,[1])$. Therefore the images of $(0:1:0)$ on $C_X$ and
$X_E(3)$ should both correspond to the class $(E,[1])$.

From Theorem 2.6, the point
corresponding to $E$ on $X_E(3)$ is the point at infinity $(\lambda:\mu)=(1:0)$. Thus, the point corresponding to
$E$ on $C_X$ should be one of the preimages of $(1:0)$, namely $(\lambda:\mu:y)=(1:0:\pm 1)$.

We use the following algorithm. Compute
the 3-isogeny $f: X_E(6) \to J_X$. Then compute an isomorphism
$g: J_X \to C_X$ where $O \in J_X$ is sent to one of
$(1:0:\pm 1)$. It does not matter
which one we pick because if the class $(E,[1])$ corresponds to $(1:0:1)$ on $C_X$, then we compose $\chi_3$ with the automorphism $[-1]$ on $X_E(6)$. We would then have $(E,[1])$ corresponding to $(1:0:-1)$ on $C_X$. We take $(1:0:1)$ for convention.
Finally, take
the quotient map $C_X \to X_E(3)$ by using Proposition 3.4.

Note that Lemma 3.10 is proved
in terms of function fields and thus it is possible that $X_E(6) \to J_X$ differs from the 3-isogeny by an
automorphism $\chi'$ of $J_X$. As we fix the convention that the point of infinity should always
corresponds to $(E,[1])$ and $\chi'$ must fix $O$. There are six such possible automorphisms, generated by $[\zeta_6]$. Since $X_E(6)$ itself is a curve of $j$-invariant 0, so we have the following diagram commutes
\begin{center}
$\begin{CD}
X_E(6) @>[\zeta^i_6]>> X_E(6)\\
@VV\chi_3 V @VV\chi_3 V\\
J_X@>[\zeta^i_6] >> J_X
\end{CD}$
\end{center}
Finally it does not matter if we compose the resulting morphism with any automorphism on $X_E(6)$ because it only changes the coordinate of $X_E(6)$.

Thus, following the above argument and computational detail we conclude
\begin{theorem}
Let $E: y^2=x^3+ax+b$ be an elliptic curve. We take an affine piece of $X_E(6): y^2=x^3+\Delta_E$, where $\Delta_E=-16(4a^3+27b^2)$. Then the morphism $X_E(6) \to X_E(3)$ is given by
$$(x,y) \mapsto v/3$$
where
$$v=\frac{-\frac{1}{6}x^3y - 18bx^3+ \frac{4\Delta_E}{3}y}{x^4 + 12ax^3 + 4\Delta_Ex}.$$
\end{theorem}
\begin{proof} The proof basically follows from the argument above. We give some computational detail. Note $J_X$ has equation $y^2=x^3-27\Delta_E$. We obtain the $3$-isogeny
$$f:X_E(6) \to J_X,(x,y) \mapsto \left(\frac{x^3+4\Delta_E}{x^2},\frac{x^3y-8\Delta_Ey}{x^3}\right)$$
and the isomorphism
$$g:J_X \to C_X,  (x,y) \mapsto \left(\frac{-\frac{y}{6}-18b}{x+12a},\frac{\frac{1}{18}x^3 + ax^2 - \frac{1}{36}y^2 - 6by - 48a^3 - 324b^2}{(x+12a)^2}\right)$$
which sends $O$ to $(1:0:1)$. Taking composition we obtain the map $\chi_3$. Finally use Proposition 3.4 we have the desired morphism.
\end{proof}
Thus, the families of elliptic curves parameterized by $Y_E(6)$ are $\mathcal{E}_{\lambda,\mu}$ in Theorem 2.6 with $(\lambda:\mu)$ in the image of $Y_E(6) \to Y_E(3)$.
\begin{remark}
\begin{enumerate}
\item We did not use the observation that $X,Y$ have rational points as we will see later the analogue of these curves in the reverse case do not always have rational points..\\
\item Compared to the results in [R], the result we got in Theorem 3.11 is a bit simpler and works for all $j$-invariant.
\end{enumerate}
\end{remark}
\section{Explicit Families of Reverse 6-Congruent Elliptic Curves}
We now use a similar approach to compute a model for $X^{-}_E(6)$.
\subsection{Setup}
We establish a commutative diagram as in the previous section as follows:
\begin{center}
$\begin{CD}
X^{-}_E(6) @>\chi^{-}_2>>Y^{-} @>\rho^{-}_2>> X^{-}_E(2)\\
@VV\chi^{-}_3V @VV \alpha^{-} V \\
X^{-} @>\beta^{-}>> Z^{-} \\
@VV \rho^{-}_3V\\
X^{-}_E(3)
\end{CD}$
\end{center}
Note that $X^{-}_E(2)$ is the same as $X_E(2)$ and $H,H'$ are the same as in the previous section and we define the quotient maps in a similar way as in the previous section. We now compute models for $X^{-},Y^{-}$ and $Z^{-}$.
\begin{proposition} Let $E:y^2=x^3+ax+b$ be an elliptic curve. Then a model of $X^{-}$ in weighted projective space is given by
$$C_{X^{-}}: y^2=\Delta_E(a\lambda^4+6b\lambda^3\mu-2a^2\lambda^2\mu^2-2ab\lambda\mu^3+(-a^3/3-3b^2)\mu^4).$$
By taking affine pieces, the morphism $C_{X^{-}} \to X^-_E(3)$ is
$$C_{X^{-}} \to X^{-}_E(3), \quad (\lambda,y) \mapsto \lambda/3, \text{ if }\mu=1 \text{ and } (1:0:\pm \sqrt{a\Delta_E}) \mapsto (1:0).$$
\end{proposition}
\begin{proof} The proof is similar to that of Proposition 3.4. By using the results for $X^{-}_E(3)$ in Theorem 2.6, a model for $X^{-}$ in weighted projective space
$(\lambda:\mu:y)$ is
$$X^{-}: y^2\Delta_E=\frac{2^{36} 3^{36} a^2}{\Delta^4_E}h(\lambda,\mu)^3,$$
where
$$h(\lambda,\mu)=a\lambda^4 + 2b\lambda^3\mu - \frac{2}{9}a^2\lambda^2\mu^2 - \frac{2}{27}ab\lambda\mu^3 + (-\frac{1}{243}a^3 - \frac{1}{27}b^2)\mu^4,$$
Putting $X^{-}$ in the form
\begin{align*}
\left(\frac{y\Delta^3_E}{2^{18}3^{16}ah(\lambda,\mu)}\right)^2
=\Delta_E&((3\lambda)^4+6b(3\lambda)^3\mu-2a^2(3\lambda^2)\mu^2-2ab(3\lambda)\mu^3\\
+&(-a^3/3-3b^2)\mu^4),\\
\end{align*}
we see that it is birational to $C_{X^{-}}$ by
$$(\lambda:\mu:y) \mapsto \left(3\lambda: \mu: \frac{y\Delta^3_E}{2^{18}3^{16}ah(\lambda,\mu)}\right).$$

The last part follows from taking compositions $C_{X^-} \to X^- \to X^{-}_E(3)$ where $X^- \to X^{-}_E(3)$ is given by $(\lambda:\mu:y) \mapsto (\lambda:\mu)$.
\end{proof}
Let $F$ be an elliptic curve which is $2$-congruent to $E$. Observe that if $F$ and $E$ are also reverse 3-congruent, then the product of their discriminants is a rational cube. In fact, this is immediate from the computation of $X^{-}$ in Proposition 4.1. Thus by exactly the same argument as we used in Proposition 3.5 we conclude that $Y$ is the modular curve which parameterizes families of elliptic curves which are $2$-congruent to $E$ whose
discriminant multiplied by $\Delta_E$ is a rational cube. Therefore,
\begin{proposition}
Let $E: y^2=x^3+ax+b$ be an elliptic curve. Then a model of $Y^{-} \subset \mathbb{P}^2$ is given by
$$C_{Y^{-}}: y^3=\Delta_E(v^3+au^2v+bu^3)$$
and the morphism $C_{Y^{-}}  \to X_E(2)$ is $(u:v:y) \mapsto (u:v)$.
\end{proposition}
\begin{proof} Using the explicit equations for $X_E(2)$ as in Theorem 2.5 and the argument above, we have
$$y^3=\Delta_E \Delta_{F_{u,v}}$$
which gives
$$\left(\frac{9\Delta_E(v^3+au^2v+bu^3)}{y}\right)^3=(v^3+au^2v+bu^3)^2\Delta_E.$$
This is birational to $C_{Y^{-}}$ by
$$(u:v:y) \mapsto (u:v:9\Delta_E(v^3+au^2v+bu^3)/y).$$
\end{proof}
The proof of Lemma 3.6, 3.7 and 3.10 did not use any property of direct congruence and so the same results still hold in the reverse case. In particular, an analogue of Remark 3.8 shows $X^{-}_E(6)$ and $Y^-$ have the same Jacobian. A direct computation using the model $C_{Y^-}$ for $Y^-$ shows
\begin{corollary} The Jacobian of $X^{-}_E(6)$ is
$$y^2z=x^3+\frac{1}{\Delta_E}z^3.$$
\end{corollary}
As $Z^-$ is the Jacobian of $X^-$, a direct computation using the model $C_{X^-}$ for $X^-$ shows
\begin{proposition}
A model of $Z^-$ is given by
$$y^2z=x^3-\frac{27}{\Delta_E}z^3.$$
\end{proposition}
\subsection{Explicit Model}
$\quad$We will compute a model for $X^{-}_E(6)$. In this previous section, we compute a model
of $X_E(6)$ by the fact it always admits a rational point. However, this is not the case in the reverse case. We give an example where $X^{-}$ has no rational point which then implies $X^-_E(6)$ has no rational point.
\begin{example} Let $a=1,b=0$ then the curve $C_{X^{-}}$ is isomorphic to
$$3y^2=-3x^4+6x^2+1.$$
This is not locally soluble at 3, and so has no rational solution.

\end{example}
We see from the commutative diagram $X^{-}$ is a 2-covering of $Z^{-}$, $Y^{-}$ is a 3-isogeny covering of $Z^-$ and $X^{-}_E(6)$ is the fiber product of $X^{-}$ and $Y^{-}$ over $Z^-$ with respect to the quotient maps.
We construct curves $X^{-}_1,Z^{-}_1$ such that $X^{-}_1$ is a 2-covering of $Z^{-}_1$ and $Y^{-}$ is a
3-covering of $Z^{-}_1$, then we apply the algorithms in [F3], which gives a way to compute a model for the
6-covering of an elliptic curve as the fiber product of a 2-covering and a 3-covering.

As we require
$X^{-}_E(6)$ to be a 6-covering of $Z^{-}_1$, so we take $Z^{-}_1$ to be the Jacobian of $X^{-}$
$$Z^{-}_1: y^2z=x^3+\frac{1}{\Delta_E} z^3.$$
Then we need to find a curve $X^{-}_1$ such that the following diagram commutes
\begin{center}
$\begin{CD}
X^{-}_E(6) @>>> Y^{-}\\
@VVV @VVV\\
X^{-}_1@>>> Z^{-}_1
\end{CD}$
\end{center}
Since $Z^{-}$ is already a 3-isogeny covering of $Z^{-}_1$ we can compute $X^{-}_1$ by making the following diagram commute
\begin{center}
$\begin{CD}
X^{-}_E(6) @>>> Y^{-}\\
@VVV @VVV\\
X^{-} @>>> Z^{-}\\
@VVV @VVV\\
X^{-}_1@>>> Z^{-}_1
\end{CD}$
\end{center}
where the map $Z^{-} \to Z^{-}_1$ is geometrically the dual isogeny of $Y^{-} \to Z^{-}_1$.
\begin{lemma} A model of $X^{-}_1$ is given by the equation
$$-3y^2=\Delta_E(a\lambda^4+6b\lambda^3\mu-2a^2\lambda^2\mu^2-2ab\lambda\mu^3+(-a^3/3-3b^2)\mu^4).$$
Note that this only differs from the equation of $C_{X^-}$ by a factor of $-3$ on the left hand side.
\end{lemma}
\begin{proof} A direct computation shows the Jacobian of the above curve is $Z^-_1$ by using formulae in 3.1 of [AKM$^3$P]. The map $\hat{\chi}^{-}_3: X^{-} \to X^{-}_1$ is geometrically a 3-isogeny. Thus, it can be constructed by $\hat{\chi}^{-}_3=I^{-1}_2 f I_1$, where $f: Z^{-} \to Z^{-}_1$ is the $3$-isogeny and the isomorphisms $I_1,I_2$, defined over $\bar{\mathbb{Q}}$ are the ones which satisfy the following commutative diagrams:
\begin{center}
$\begin{CD}
X^{-} @>\pi_2 >> Z^{-}\\
@VV I_1 V @|\\
Z^{-} @>[2]>> Z^{-}\\
\end{CD}$
\end{center}
\begin{center}
$\begin{CD}
X^{-}_1 @>\pi'_2 >> Z^{-}_1\\
@VV I_2 V @|\\
Z^{-}_1 @>[2]>> Z^{-}_1\\
\end{CD}$
\end{center}
where $\pi_2, \pi'_2$ are the 2-covering maps defined over $\mathbb{Q}$. Then the resulting $\hat{\chi}^{-}_3$
is geometrically a 3-isogeny and also satisfies the commutative diagram
\begin{center}
$\begin{CD}
X^{-} @>\pi_2 >> Z^{-}\\
@VV \hat{\chi}^{-}_3 V @VV fV\\
X^{-}_1 @>\pi'_2>> Z^{-}_1\\
\end{CD}$
\end{center}
because
$$\pi'_2\hat{\chi}^{-}_3=[2](I_2I^{-}_2)(fI_1)=([2]f)I_1=f([2]I_1)=f\pi_2$$

Note that $I_1$ is not unique, for if we replace $I_1$ by $T_i I_1$ where $T_i$ is translation by 2-torsion point then the resulting map will still satisfy the commutative diagram and similarly $I_2$ is not unique. Finally, we want the map $\hat{\chi}^{-}_3$ to be defined over $\mathbb{Q}$ and this is equivalent to
$$I_2\sigma(I^{-}_2) f \sigma(I_1)I^{-}_1=f, \forall \sigma \in G_\mathbb{Q}$$

Explicitly, the maps $\sigma(I_1)I^{-}_1$ and
$I_2 \sigma(I^{-}_2)$ are both translation by 2-torsion point. As $f$ is a 3-isogeny which induces
an isomorphism between $Z^{-}[2]$ and $Z^{-}_1[2]$, we conclude that this holds if we pick the suitable choices
for $I_1$ and $I_2$ in the sense we match up the cocycles (which means we match up the 2-torsion points induced by the cocycles). One then checks the
resulting $\hat{\chi}^-_3$ maps $X^{-}$ onto $X^{-}_1$.

\end{proof}
Now we work through the algorithms in [F3]. Following the notation in [F3], we firstly embed our 2-covering
in $\mathbb{P}^5$ with coordinates $(x_1:\ldots:x_6)$ in the way
$$\begin{pmatrix} x_1&x_2&x_3 \\ x_4&x_5&x_6 \end{pmatrix}$$
such that the action of translation by 2-torsion points is multiplication by certain 2-by-2 matrices on the left and the action of translation by 3-torsion points is multiplication by certain 3-by-3 matrices on the right.
The algorithm to do this is described in Proposition 5.1 [F3].

Fix an elliptic curve $E:y^2=x^3+ax+b$ and the 2-covering we have here is the one we
got from Lemma 4.6, namely
$$-3y^2=\Delta_E(a\lambda^4+6b\lambda^3\mu-2a^2\lambda^2\mu^2-2ab\lambda\mu^3+(-a^3/3-3b^2)\mu^4).$$
In the notation of Proposition 5.1 [F3], we have
$$U(\lambda,\mu)=-3\Delta_E(a\lambda^4+6b\lambda^3\mu-2a^2\lambda^2\mu^2-2ab\lambda\mu^3+(-a^3/3-3b^2)\mu^4).$$
The invariant $H$ is given in Definition 4.1 [F3], which is
$$H=\frac{1}{3}\det\left(\frac{\partial^2 U}{\partial \lambda \partial \mu}\right).$$
Thus we can compute the image $A_{2,3}$ as in Proposition 5.1 [F3]. Explicitly, we have
$$A_{2,3}=\begin{pmatrix} -9\frac{\partial H}{\partial \mu} & -3\frac{\partial U}{\partial \mu} & \lambda y\\
9\frac{\partial H}{\partial \lambda} & 3\frac{\partial U}{\partial \lambda} & \mu y \end{pmatrix}
=\begin{pmatrix} A_1&A_2&A_3\\A_4&A_5&A_6 \end{pmatrix}.$$

It is well-known that  a genus one curve in $\mathbb{P}^5$ embedded by a complete linear system of degree 6 is given by the intersection of 9 quadrics, and we are free to replace each of them by some invertible linear transformation of them. To work out the image,
take a polynomial ring of 6 variables and let $T$ be the set of all 21 monomials of degree 2. Then evaluate each monomial in $T$ at entries
$A_i$ above and reduce these modulo $I$,
where $I=\langle y^2-U(\lambda,\mu) \rangle$. Then we denote the set of the resulting 21 monomials by $T'$. Also denote the set of all the monomials in
$y,\lambda,\mu$ which appear in $T'$ by $S$. Finally we work out the kernel of the matrix, whose $i,j$-th entry is the coefficient of the $j$-th monomial in $S$ in the $i$-th monomial $T'$.

We give our resulting quadrics after taking some suitable invertible linear transformation (to tidy up the equations of the quadrics)
\begin{align*}
q_1&=-x_1x_4+8aD^3x_2x_5-3D^5x_3x_6+12bD^3x^2_5,\\
q_2&=x_1x5+72bDx^2_2+x_2x_4-32a^2Dx_2x_5+24aD^3x_3x_6\\
&-24abDx^2_5+36bD^3x^2_6,\\
q_3&=-72bDx_1x_5-72bDx_2x_4-576a^2D^4x3^2-1728abD^4x_3x_6-x^2_4\\
&+32a^2Dx_4x_5-8aD^3x_5^2-1296b^2D^4x^2_6,\\
q_4&=-24Dax_2x_3-36bDx_2x_6-36bDx_3x_5+x_4x_6+8a^2Dx_5x_6,\\
q_5&=-3x^2_1+9D^5x^2_3-ax^2_4-D^2x_4x_5+3aD^5x^2_6,\\
q_6&=24aDx^2_2+72bDx_2x_5+18D^3x^2_3-x_4x_5-8a^2Dx^2_5+6aD^3x^2_6,\\
q_7&=x_1x_6+72bDx_2x_3-16a^2Dx_2x_6+x_3x_4-16a^2Dx_3x_5-24abDx_5x_6,\\
q_8&= -3x_1x_2+36D^3ax^2_3+108bD^3x_3x_6-ax_4x_5+D^2x^2_5/2-12a^2D^3x^2_6,\\
q_9&=-3x_1x_3-ax_4x_6+D^2x_5x_6/2,\\
\end{align*}
where we use the notation $D=\Delta_E=-16(4a^3+27b^2)$.

Then we follow the procedure in section 6 [F3], and the curve $X^{-}_E(6)$ is given by intersection of 9-quadrics in variables $X_1,\ldots,X_6$ with the relation
$$\begin{pmatrix} X_1&X_2&X_3\\X_4&X_5&X_6 \end{pmatrix}=\begin{pmatrix} x_1&x_2&x_3 \\ x_4&x_5&x_6
\end{pmatrix} g$$
where $g$ is a change of coordinate matrix to put $C_{Y^-}$ in Weierstrass form. The formula of $g$ we use here can be found on page 24 in [F3]. The explicit algorithm to compute the flex matrix is given in Algorithm 6.2 and 6.3. Following the procedure, we obtain the inverse of the flex matrix (as in fact
this is more useful for our calculation)
$$g^{-1}=\begin{pmatrix} -648D^2a\alpha-972bD^2 &0&1\\972bD^2\alpha-216a^2D^2&0&\alpha\\0&-18D\alpha^2-6aD&0
\end{pmatrix},$$
where $\alpha$ is a root of $x^3+ax+b=0$.

Finally, as is described on page 24 in [F3], the new quadrics have coefficients in an extension of $\mathbb{Q}$
but the space they span has a basis with coefficients in $\mathbb{Q}$. So what we do is compute $x_i$ in terms
of $X_i$ and then substitute these into the quadrics $q_i$. Thus the new quadrics have coefficients in
$\mathbb{Q}(\alpha)$. But as they have a basis over $\mathbb{Q}$ we collect the coefficients of $\alpha^2,\alpha,1$ and the vanishing of them will give 3-quadrics.
Thus, we obtain 27-quadrics with coefficients in $\mathbb{Q}$ and we check they span a 9-dimensional vector space.

We use our resulting quadrics $q_i$ given above and apply this method then it can be checked by direct
computation that the 27-quadrics obtained after twisted by the flex matrix span a 9-dimensional vector space, and
a basis can be given by the following 9-quadrics (again after some invertible linear transformation on the basis):
\begin{align*}
s_1&=-6x_1x_5 + 24ax_1x_6 - 6x^2_2 + 24ax_2x_3 - 6x_2x_4 + 24ax_3x_4\\
 &+ 72bx_4x_6 + 2ax^2_5 + 8a^2x_5x_6 +Dx^2_6,\\
s_2&=-6x_1x_3 + x_2x_5 + 2ax_2x_6 - 36bx^2_3 + 2ax_3x_5 + 16a^2x_3x_6 + x_4x_5\\
& + 2ax_4x_6 + 12abx^2_6,\\
s_3&=12ax_1x_3 + 18bx_1x_6 + 18bx_3x_4 - 2ax_4x_5 - 4a^2x_4x_6 + 3bx^2_5,\\
s_4&=-12ax_2x_3 - 18bx_2x_6 - 18bx_3x_5 - 3x^2_4 - ax^2_5 + 4a^2x_5x_6,\\
s_5&=3x^2_2 - 48a^2x^2_3 - 144abx_3x_6 - 36bx_4x_6 + ax^2_5 - 8a^2x_5x_6 + 16a^3x^2_6,\\
s_6&=-3x_1x_4 + 18bx_2x_3 - ax_2x_5 - 4a^2x_2x_6 - 4a^2x_3x_5 - 6abx_5x_6,\\
s_7&=-108bx_1x_3 + 6ax^2_2 - 24a^2x_2x_3 + 18bx_2x_5 - 36abx_4x_6 - 2a^2x^2_5\\
& - 8a^3x_5x_6-aDx^2_6,\\
s_8&=3x_1x_2 - 72abx^2_3 - 216b^2x_3x_6 + ax_4x_5 + 8a^2x_4x_6 - 12abx_5x_6 + 24a^2bx^2_6,\\
s_9&=36x^2_1 + 12ax^2_2 + 12ax^2_4 + 4a^2x^2_5 + Dx_5x_6.\\
\end{align*}
This gives a model for $X^{-}_E(6)$ which proves the first part of Theorem 1.1.
\subsection{Explicit Equations And Examples}
$\quad$We now compute the morphism from $X^{-}_E(6)$ to $X^{-}_E(3)$ using the model we obtained.
In [F3], by Theorem 2.5 (or page 24), the map from the 6-covering to the 3-covering is recovered by taking
the 2-by-2 minors of the matrix
$$\begin{pmatrix} x_1&x_2&x_3\\x_4&x_5&x_6 \end{pmatrix},$$
where our curve $X^{-}_E(6)$ is the intersection of 9 quadrics in $\mathbb{P}^5$ with coordinates $(x_1:\ldots:x_6)$.
Thus, this gives a map from $X^{-}_E(6)$ to $C_{Y^{-}}$ by our construction:
$$\chi^{-}_2(x_1:\cdots:x_6)=(u:v:y),u=x_2x_6-x_3x_5, v=x_3x_4-x_1x_6, y=x_1x_5-x_2x_4.$$

For any point $P=(x_1:\cdots:x_6) \in X^{-}_E(6)$, take $\chi^{-}_2(P)$ defined above and then
by Proposition 4.2 we obtain the morphism
$$X^{-}_E(6) \to X_E(2), \quad (x_1:\cdots:x_6) \mapsto (u:v)=(x_2x_6-x_3x_5: x_3x_4-x_1x_6)$$
and compose this with $X_E(2) \to X(1)$ we obtain
the $j$-map $X^{-}_E(6) \to X(1)$ and let $j(P)$ be the image of $P$.

Suppose $(\lambda:\mu)$ is the image of $P$ on $X^-_E(3)$ then we must have $j(\lambda:\mu)=j(P)$ where
$j(\lambda:\mu)=j(\mathcal{E}'_{\lambda,\mu})$ with $\mathcal{E}'_{\lambda,\mu}$ given in Theorem 2.6. Therefore,
\begin{theorem} The forgetful morphism $X^{-}_E(6) \to X^{-}_E(3)$ is given by
$$(x_1:\cdots:x_6) \mapsto (x_3/3:x_6).$$
\end{theorem}
\begin{proof} We implement the above algorithm where $j(\mathcal{E}'_{\lambda,\mu})$ is a degree 12 homogenous polynomial in $\lambda,\mu$. Let $L=\mathbb{Q}(a,b)$.
If $(x_1:\cdots:x_6)$ is a $L$-rational point then the image must also be $L$-rational.
Thus, we look for a $L$-rational point at which the degree 12 polynomial vanishes. A direct computation shows that the polynomial
is equal to zero at the point $(\lambda:\mu)=(x_3/3:x_6)$ (recall the point $(x_1:\cdots:x_6)$ satisfies the equation
of $X^{-}_E(6)$).
\end{proof}
\begin{remark}
One can also start with $L$ and define the function field of $X^{-}_E(6)$ over $L$ and then factorize the polynomial to check there is a unique linear factor which agrees with our computation above.
\end{remark}
Theorem 4.7 shows the families of elliptic curves parameterized by $Y^{-}_E(6)$ are $\mathcal{E}'_{x_3/3,x_6}$ as Theorem 2.6 with
$(x_1:\ldots:x_6) \in Y^{-}_E(6)$.
\begin{example}
Let $t \in \mathbb{Q}$ be any rational number and let $a=-\frac{8}{27}t^2, b=\frac{64}{729}t^3$.
Then there is a rational point on $X^{-}_E(6)$ given by
$$x_1=-\frac{2^9t^4}{2187},x_2=\frac{2^6t^3}{243},x_3=\frac{2}{9}t,x_4=-\frac{2^6t^3}{243}, x_5=\frac{2^5}{27}t^2,x_6=1.$$
Then the image of this point on $X^-_E(3)$ is $(\frac{2}{27}t:1)$. Taking $t=\frac{9}{2}$ we obtain a pair of non-isogenous elliptic curves which are reverse $6$-congruent
\begin{align*}
E:& y^2=x^3-6x+8 &\quad 6912v1\\
F:& y^2=x^3 -216x + 1728 & \quad 6912p1\\
\end{align*}
Also, one computes that the rank of both $X_E(6)$ and $X^{-}_E(6)$ are positive.
\end{example}

\begin{example}
The curves in the previous examples are quadratic twists to each other. In this example, we give a list of infinitely many values of $j$-invariant such that for any curve with $j$-invariant in the list, $X^{-}_E(6)$ has at least one $\mathbb{Q}$-rational point.
For each $j \neq 0,1728$ there exists a unique value $t$ such that the elliptic curve $E_t: y^2=x^3+tx+t$ has $j$-invariant $j$. For $(u:v) \in \mathbb{P}^1(\mathbb{Q})$, we take
$$t=-\frac{27}{8}\frac{(u-v)^3(u+v)^3}{(u^2-uv+v^2)^2(u^2-uv-\frac{1}{2}v^2)}.$$
Then by direct computation we obtain a rational point on $X^{-}_{E_t}(6)$:
\begin{align*}
x_1&=t^3(\frac{2}{3}u^7 - \frac{7}{2}u^6v + \frac{15}{2}u^5v^2 - \frac{26}{3}u^4v^3 + \frac{11}{2}u^3v^4 - \frac{3}{2}u^2v^5 -
    \frac{1}{3}uv^6)\\
x_2&=\frac{3t^2}{4}(-u^7 + u^6v + 4u^5v^2 - 2u^4v^3 - 5u^3v^4 + u^2v^5 +2uv^6)\\
x_3&=\frac{3t}{16}(u^7 + u^6v - 3u^5v^2 - 3u^4v^3 + 3u^3v^4 +
    3u^2v^5 - uv^6 - v^7)\\
x_2&=-\frac{3t^2}{4}(-u^7 + u^6v + 4u^5v^2 - 2u^4v^3 - 5u^3v^4 + u^2v^5 +2uv^6)\\
x_5&=t^2(u^7 - \frac{7}{2}u^6v + \frac{7}{2}u^5v^2 - \frac{7}{2}u^3v^4 + \frac{7}{2}u^2v^5 - uv^6)\\
x_6&=\frac{t}{8}(-2u^7 + 3u^6v +u^5v^2 - 5u^4v^3 + 4u^3v^4 + u^2v^5 -3uv^6 + v^7)\\
\end{align*}
\end{example}
In particular, the above example shows that
\begin{corollary}
There are infinitely many pairs of non-isogenous elliptic curves which are reverse $6$-congruent.
\end{corollary}
\subsection{A Simpler Birational Model}
Finally we derive a much simpler birational model of $X^{-}_E(6)$.
\begin{proposition} The forgetful morphism from $X^{-}_E(6) \to C_{X^{-}}$ is given by
$$(x_1,x_2,x_3,x_4,x_5) \mapsto (x_3,\frac{-x_1x_5+x_2x_4}{2}),$$
where we take an affine piece $x_6=1$ on $X^{-}_E(6)$ and $\mu=1$ on $C_{X^{-}}$
\end{proposition}
\begin{proof} Using the morphisms $X^-_E(6) \to X^-_E(3)$ and $X^- \to X^-_E(3)$ we conclude that the morphism $X^{-}_E(6) \to X^{-}$ must have the form
$$(x_1,x_2,x_3,x_4,x_5) \mapsto (x_3, y)$$
where $y$ satisfies the equation
$$y^2=\Delta_E(ax^4_3+6bx^3_3-2a^2x^2_3-2abx_3+(-a^3/3-3b^2)).$$
A direct computation shows that $y=\pm \frac{-x_1x_5+x_2x_4}{2}$ satisfy
the condition.
It does not matter which one we pick because it differs from the other one by applying the
$[-1]$ map on $X^{-}_E(6)$.
\end{proof}
We can now simplify our defining equations for $X^{-}_E(6)$.
\begin{corollary} The curve $X^{-}_E(6)$ is birational to the following curve defined by two equations in $\mathbb{A}^3$ with coordinates $(x,y,z)$:
\begin{align*}
f=&z^3-(36ax^2+12a^2)z+216bx^3-144a^2x^2-216abx\\
&-(16a^3+216b^2)+y(64abx+96b^2)27/\Delta_E,\\
g=&y^2-\Delta_E(ax^4+6bx^3-2a^2x^2-2abx+(-a^3/3-3b^2)).\\
\end{align*}
with the forgetful morphisms
$$X^{-}_E(6) \to C_{X^{-}}: (x,y,z) \mapsto (x,y),~ C_{X^{-}} \to X^{-}_E(3): (x,y) \mapsto x/3.$$
\end{corollary}
\begin{proof}
The equation for $g$ is essentially the same as the defining equation for $X^{-}$. If we take $z=x_5,x=x_3,y=(-x_1x_4+x_2x_3)/2$, then we obtain a relation defined by $f$ above, Also $x,y$ satisfy the equation defined by $g$, using Proposition 4.12. Since $X^{-}_E(6)$ is a curve and hence the function field
has transcendental degree $1$ and thus the field generated by $x,z,y$ is the same as the function field of
$X^{-}_E(6)$. Therefore, we conclude that the curve defined by equations $f$ and $g$ is birational to
$X^{-}_E(6)$. Finally, by using Theorem 4.7 and Proposition 4.12 we obtain the desired forgetful morphisms.
\end{proof}

{\bf Zexiang Chen}\\
Department of Pure Mathematics and Mathematical Statistics,\\
Centre for Mathematical Sciences,\\
University of Cambridge.\\
e-mail: zc231@cam.ac.uk

\end{document}